\documentclass{amsart}
\pdfoutput=1

\usepackage[utf8]{inputenc}
\usepackage{lmodern}

\usepackage{amssymb}
\usepackage{amsfonts}

\usepackage[lite]{amsrefs}

\usepackage{mathtools}
\usepackage{enumitem}
\usepackage{microtype}

\usepackage[pdftitle={A universal coefficient theorem for actions of finite groups on C*-algebras},
pdfauthor={Ralf Meyer, George Nadareishvili},
pdfsubject={Mathematics}
]{hyperref}

\numberwithin{equation}{section}

\newtheorem{theorem}{Theorem}[section]
\newtheorem{lemma}[theorem]{Lemma}
\newtheorem{proposition}[theorem]{Proposition}
\newtheorem{corollary}[theorem]{Corollary}

\theoremstyle{definition}
\newtheorem{definition}[theorem]{Definition}

\theoremstyle{remark}
\newtheorem{remark}[theorem]{Remark}
\newtheorem{example}[theorem]{Example}

\newcommand*{\Cst}{\textup C^*}
\newcommand*{\Cont}{\textup C}
\newcommand*{\nb}{\nobreakdash}
\newcommand{\into}{\rightarrowtail}
\newcommand{\prto}{\twoheadrightarrow}

\newcommand{\C}{\mathbb{C}}
\newcommand{\K}{\mathrm{K}}
\newcommand{\Z}{\mathbb{Z}}
\newcommand{\KK}{\mathrm{KK}}
\newcommand{\ideal}{\mathfrak{I}}

\newcommand{\N}{\mathbb N}
\newcommand{\Comp}{\mathbb K}
\newcommand{\Mat}{\mathbb M}

\DeclarePairedDelimiter{\abs}{\lvert}{\rvert}

\DeclareMathOperator{\Hom}{Hom} 
\DeclareMathOperator{\Ext}{Ext} 

\newcommand*{\defeq}{\mathrel{\vcenter{\baselineskip0.5ex \lineskiplimit0pt
                     \hbox{.}\hbox{.}}}%
                     =}

\DeclareMathOperator{\Res}{Res}
\DeclareMathOperator{\Ind}{Ind}

\title[A UCT for actions of finite group on $\Cst$\nb-algebras]{A
  universal coefficient theorem for actions of finite groups on
  $\Cst$\nb-algebras}

\author{Ralf Meyer}
\email{rmeyer2@uni-goettingen.de}
\address{Mathematisches Institut\\
  Universit\"at G\"ottingen\\Bunsenstra\ss e 3--5\\
  37073 G\"ottingen\\Germany}

\author{George Nadareishvili}
\email{giorgi.nadareishvili@kiu.edu.ge}
\address{School of Mathematics\\
Kutaisi International University\\Akhalgazrdoba Ave.~Lane~5/7\\
4600 Kutaisi\\Georgia 
}


\begin{document}

\begin{abstract}
  The equivariant bootstrap class in the Kasparov category of
  actions of a finite group~$G$ consists of those actions that are
  equivalent to one on a Type~I $\Cst$\nb-algebra.  Using a result
  by Arano and Kubota, we show that this bootstrap class is already
  generated by the continuous functions on~$G/H$ for all cyclic
  subgroups~$H$ of~$G$.  Then we prove a Universal Coefficient
  Theorem for the localisation of this bootstrap class at the
  group order~$\abs{G}$.  This allows us to classify certain
  $G$\nb-actions on stable Kirchberg algebras up to cocycle
  conjugacy.
\end{abstract}

\maketitle

\section{Introduction}

The Kirchberg--Phillips classification of nuclear, simple, purely
infinite, separable, stable \(\Cst\)\nb-algebras
(see~\cite{Phillips:Classification}) may be be split into two parts.
The first, analytic part shows that two such \(\Cst\)\nb-algebras
are isomorphic once they are \(\KK\)-equivalent.  The second,
topological part shows that they are \(\KK\)-equivalent once they
belong to the bootstrap class and have isomorphic \(\K\)\nb-theory.
In addition, any pair of \(\Z/2\)-graded Abelian groups is the
\(\K\)\nb-theory of some such \(\Cst\)\nb-algebra.  Both topological
statements follow from the Universal Coefficient Theorem, which
computes \(\KK(A,B)\) in terms of \(\K_*(A)\) and \(\K_*(B)\)
provided~\(A\) is in the bootstrap class.

There has been some recent progress on dynamical analogues of this
classification, where the aim is to classify certain group actions.
On the analytic side, the dynamical Kirchberg--Phillips theorem by
Gabe and Szab\'o in~\cite{Gabe-Szabo:Dynamical_Kirchberg} says that
two pointwise outer actions of a discrete, countable, amenable group
on stable Kirchberg algebras are cocycle conjugate if and only if
they are equivalent in the equivariant Kasparov category~\(\KK^G\).
On the topological side, Manuel K\"ohler~\cite{Koehler:Thesis}
showed for a finite cyclic group of prime order that isomorphism
classes of objects in the equivariant bootstrap class are in
bijection with isomorphism classes of ``exact'' modules over a
certain ring.  The relevant ring is, however, rather complicated.
In~\cite{Meyer:Actions_Kirchberg}, the classification is worked out
in detail in the special case when the \(\K\)\nb-theory is a cyclic
group.  In addition, the case when~\(G\) is torsion-free amenable is
also treated in~\cite{Meyer:Actions_Kirchberg}: actions of such
groups are equivalent to locally trivial bundles over the
classifying space of~\(G\).

The classification in~\cite{Meyer:Actions_Kirchberg} for the cyclic
group of prime order~\(p\) is remarkably subtle in general, but it
becomes rather simple if~\(p\) is invertible in the ring
\(\KK^G(A,A)\).  Here we generalise this easier part of the
classification to an arbitrary finite group~\(G\), assuming that the
group order~\(\abs{G}\) is invertible in \(\KK^G(A,A)\).  Our result
is related to recent independent work by Bouc, Dell'Ambrogio and
Martos~\cite{Bouc-DellAmbrogio-Martos:Splitting}.  They prove that
the localisation of the bootstrap class in \(\KK^G\)
at~\(\mathbb{Q}\) is semisimple and compute \(\KK^G(A,A)\) when it
is a \(\mathbb{Q}\)\nb-vector space.  Our approach is more
elementary, using explicit polynomials for some key computations.
This allows us to prove the main result after inverting only the
group order~\(\abs{G}\).

Another difference is that~\cite{Bouc-DellAmbrogio-Martos:Splitting}
treats only the Kasparov category of \(G\)-cell algebras.  Using an
important theorem by Arano and
Kubota~\cite{Arano-Kubota:Atiyah-Segal}, we show that this
subcategory is the same as the equivariant bootstrap class
in~\(\KK^G\), that is, the subcategory of all objects that are
\(\KK^G\)-equivalent to a \(G\)\nb-action on a separable
\(\Cst\)\nb-algebra of Type~I.  As a result, our main result
classifies pointwise outer actions of finite groups on~\(A\) up to cocycle
conjugacy provided~\(A\) is a stable Kirchberg algebra, the action belongs
to the \(G\)\nb-equivariant bootstrap class, and~\(\abs{G}\) is
invertible in the ring \(\KK^G(A,A)\).  Incidentally, the results of
Arano and Kubota~\cite{Arano-Kubota:Atiyah-Segal} also imply that a
\(G\)\nb-action belongs to the \(G\)\nb-equivariant bootstrap class
if and only if the restrictions of the action belong to the
\(H\)\nb-equivariant bootstrap class for all cyclic subgroups
\(H\subseteq G\).  This implies that \(A\) and \(A\rtimes H\) for
cyclic subgroups \(H\subseteq G\) are in the bootstrap class
in~\(\KK\).  We do not know whether this necessary
condition is sufficient as well.

We end the introduction by formulating our main theorem.  Writing
down the classifying invariant needs some preparation.  Let
\(H\subseteq G\) be a cyclic subgroup and \(n\defeq \abs{H}\).  The
representation ring of~\(H\) is isomorphic to~\(\Z[z]/(z^n-1)\).
Let \(\Phi_n\in\Z[z]\) be the \(n\)th cyclotomic polynomial, whose
zeros are exactly the primitive \(n\)th roots of unity.  This
divides \(z^n-1\), so that \(\Z[z]/(\Phi_n)\) is a quotient of the
representation ring of~\(H\).  The representation ring is isomorphic
to the ring \(\KK^H_0(\C,\C)\), and the induction functor induces a
map from this to \(\KK^G_0(\Cont(G/H),\Cont(G/H))\).  Let
\(N_H \defeq \{g\in G \mid g H g^{-1} = H\}\).  This acts on~\(G/H\)
and thus on~\(\Cont(G/H)\) by right translations, with the
subgroup~\(H\) of~\(N_H\) acting trivially.  Thus we get a
homomorphism from the quotient group \(W_H\defeq N_H/H\) into the
ring \(\KK^G_0(\Cont(G/H),\Cont(G/H))\).  The group~\(W_H\) also
acts on the representation ring of~\(H\) because~\(N_H\) acts
on~\(H\) by automorphisms and inner automorphisms act trivially on
the representation ring.  The homomorphisms from the representation
ring \(\Z[z]/(z^n-1)\) and from the group~\(W_H\) to
\(\KK^G_0(\Cont(G/H),\Cont(G/H))\) are covariant and so combine to a
homomorphism on \(\Z[z]/(z^n-1)\rtimes W_H \).  Hence
\(\Z[z]/(z^n-1)\rtimes W_H \) acts on the \(\Z/2\)-graded Abelian
group \(\KK^G_*(\Cont(G/H),B) \cong \K_0(B\rtimes H)\) for any
\(G\)\nb-\(\Cst\)-algebra~\(B\).  Let
\[
  F^H_*(B) \defeq \{x\in\KK^G_*(\Cont(G/H),B) \mid \Phi_n(z)\cdot x=0\}.
\]
The graded subgroup~\(F^H_*(B)\) in \(\KK^G_*(\Cont(G/H),B)\) is
even a \(\Z/2\)-graded module over the ring
\(\Z[z]/(\Phi_n(z))\rtimes W_H\).

\begin{theorem}
  \label{the:UCT_in_introduction}
  Let~\(G\) be a finite group.  Let~\(A\) and~\(B\) be
  \(G\)-\(\Cst\)\nb-algebras.  Suppose that~\(A\) is in the
  \(G\)\nb-equivariant bootstrap class, that is, it is
  \(\KK^G\)-equivalent to an action on a Type~I \(\Cst\)\nb-algebra.
  Suppose that~\(B\) is \(\abs{G}\)-divisible in the sense that
  multiplication by~\(\abs{G}\) on \(\KK^G(B,B)\) is invertible.
  Then there is a Universal Coefficient Theorem short exact sequence
  \begin{multline*}
    \prod_{H\subseteq G \textup{ cyclic}}
    \Ext^1_{\Z[z]/(\Phi_n(z))\rtimes W_H}    \bigl(F^H_{*-1}(A),F^H_*(B)\bigr)
    \into \KK^G(A,B) \\ \prto
    \prod_{H\subseteq G \textup{ cyclic}}
    \Hom_{\Z[z]/(\Phi_n(z))\rtimes W_H}    \bigl(F^H_*(A),F^H_*(B)\bigr).
  \end{multline*}
  Here the products run over conjugacy classes of cyclic subgroups
  \(H\subseteq G\).
  If~\(M_H\) are countable \(\Z/2\)-graded modules over
  \(\Z[z,1/\abs{G}]/(\Phi_n(z))\rtimes W_H\) for all cyclic
  subgroups \(H\subseteq G\), then there is a
  \(\abs{G}\)\nb-divisible object~\(B\) in the bootstrap class in
  \(\KK^G\) with \(F^H_*(B) \cong M_H\) for all cyclic subgroups
  \(H\subseteq G\), and this is unique up to \(\KK^G\)-equivalence.
\end{theorem}

Together with the dynamical Kirchberg--Phillips theorem by Gabe and
Szab\'o, this theorem implies a classification of certain outer
group actions on Kirchberg algebras up to cocycle conjugacy.  The
proof of Theorem~\ref{the:UCT_in_introduction} is based on ideas of
Manuel K\"ohler~\cite{Koehler:Thesis}.

\section{Equivariant \texorpdfstring{$\KK$}{KK}-theory}

Let~$G$ be a second countable locally compact group.  Let $\KK^G$
denote the Kasparov category of separable $G$\nb-$\Cst$\nb-algebras.
The spatial tensor product of $\Cst$\nb-algebras with diagonal
$G$\nb-action induces a symmetric monoidal structure on $\KK^G$,
which we denote by~$\otimes$.  By $\oplus$ we denote the
$\Cont_0$-direct sum, which exists for countable collections of
$G$\nb-$\Cst$\nb-algebras; it makes $\KK^G$ an additive category
with countable coproducts.  The category \(\KK^G\) is triangulated
(see~\cite{Meyer-Nest:BC}, and see~\cite{Neeman:Triangulated} for an
introduction to triangulated categories in general).  The suspension
functor is $\Sigma\defeq \Cont_0(\mathbb R)\otimes {-}$.  It is an
involutive equivalence by Bott periodicity.  The exact triangles
come either from mapping cones of equivariant \(^*\)-homomorphisms
or from extensions of $G$\nb-$\Cst$\nb-algebras with a
$G$\nb-equivariant, completely positive contractive section (for
details see the Appendix of~\cite{Meyer-Nest:BC}).

\subsection{Functors on \texorpdfstring{$\KK^G$}{KKG}}
\label{sec:functors_KKG}

Let $H\subseteq G$ be a closed subgroup.  The \emph{restriction} of
$G$\nb-actions defines a functor \(\Res^G_H\colon \KK^G\to\KK^H\).
It preserves coproducts, is triangulated, and symmetric monoidal
(see, for instance, \cite{Meyer-Nest:BC}).  The \emph{induction
  functor} \(\Ind^G_H\colon \KK^H\to\KK^G\) is defined by
\begin{multline*}
  \Ind^G_H(A)
  \defeq
  \{ G\xrightarrow{\phi} A \mid
  \phi\text{ continuous, }
  h\phi(gh)=\phi(g) \text{ for all }g\in G,h\in H,\\
  (gH\mapsto \lVert\phi(g)\rVert)\in \Cont_0(G/H) \},
\end{multline*}
equipped with the $G$\nb-action $(g \cdot \phi)(s) = \phi(g^{-1}s)$
for $g, s \in G$.  This is a triangulated functor preserving
coproducts.
If~$G/H$ is compact, then $\Ind_H^G$ is right adjoint to $\Res^G_H$
(see \cite{Meyer-Nest:BC}*{Equation~(19)}).  If~$G/H$ is discrete,
then $\Ind_H^G$ is left adjoint to $\Res^G_H$ (see
\cite{Meyer-Nest:BC}*{Equation~(20)}).  Thus, if~$G$ is finite, then
the restriction and induction functors are adjoint both ways for any
subgroup $H\subseteq G$.

For $g\in G$ and a subgroup $H\subseteq G$, the \emph{conjugation
  functor}
\[
  {}^g({-})\colon \KK^H\to\KK^{{}^g\!H}
\]
sends a $\Cst$\nb-algebra~$A$ with an $H$-action to the same
$\Cst$\nb-algebra~\(A\) with an action of~${}^g H\defeq gHg^{-1}$ defined by
$ghg^{-1}\cdot a\defeq ha$ for $h\in H$ and $a\in A$.  The
functor of conjugation by~\(g\) is a triangulated, monoidal
equivalence: the inverse is conjugation by~$g^{-1}$.  As such, it
preserves coproducts.

This article is based on the following remarkable result of Arano and
Kubota:

\begin{theorem}[\cite{Arano-Kubota:Atiyah-Segal}*{Corollary~3.13.(1)}]
  \label{the:Arano-Kubota}
  Let~\(G\) be a finite group.  If \(A\in \KK^G\) is such that
  \(\Res_G^H(A) \cong 0\) for all cyclic subgroups \(H\subseteq G\),
  then already \(A\cong 0\) in \(\KK^G\).
\end{theorem}

In fact, the statement in~\cite{Arano-Kubota:Atiyah-Segal} is more
general.  First, it allows~\(G\) to be a compact Lie group.
Secondly, it allows \(A\) and~\(B\) to be
\(\sigma\)\nb-\(\Cst\)-algebras instead of \(\Cst\)\nb-algebras.

\subsection{Homological algebra in \texorpdfstring{$\KK^G$}{KKG}}
\label{sub:hom}

The main result of this paper is based on a Universal Coefficient
Theorem for \(\KK^G\), and this fits in the context of relative
homological algebra in a triangulated category~\(\mathfrak{T}\) (see
\cites{Beligiannis:Relative, Christensen:Ideals,
  dellAmbrogio:Cell_G, Meyer:Homology_in_KK_II,
  Meyer-Nest:Homology_in_KK}).  Accordingly, we recall some facts
from the general theory.  As in every non-abelian category, doing
homological algebra in a triangulated category requires extra
structure.  We usually specify this through a stable homological
functor.

A \emph{stable additive category} is an additive
category~$\mathfrak A$ with an auto-equivalence functor
$\Sigma\colon\mathfrak A\to \mathfrak A$, which is called the
\emph{suspension} in~\(\mathfrak{A}\).  A \emph{stable homological
  functor} from a triangulated category~\(\mathfrak{T}\) to a stable
abelian category~\(\mathfrak{A}\) is a functor
\(F\colon \mathfrak{T} \to \mathfrak{A}\) that maps exact triangles
in~\(\mathfrak{T}\) to exact sequences in~\(\mathfrak{A}\) and that
commutes with the suspension up to a natural isomorphism.
The \emph{kernel on morphisms} of~\(F\) is the family of subgroups
\(\ker F(A,B) \defeq \{\phi\in \mathfrak{T}(A,B)\mid F(\phi)=0\}\).
This is an ideal in~\(\mathfrak{T}\), and an ideal of
this form for a stable homological functor~\(F\) is called a
\emph{stable homological ideal}.  If
\(F\colon \mathfrak{T}\to \mathfrak{D}\) is a triangulated functor
to another triangulated category, then the kernel on morphisms is a
stable homological ideal as well
(see~\cite{Meyer-Nest:Homology_in_KK}).  Such homological ideals
play an important role in the localisation approach to the
Baum--Connes assembly map developed in~\cite{Meyer-Nest:BC} and will
also be crucial below.

Another homological functor~$H$ is called \emph{$\ideal$\nb-exact}
if $H(\phi)=0$ for all $\phi\in\ideal$.  An \emph{$\ideal$\nb-exact}
stable homological functor
$U\colon \mathfrak T\to \mathfrak{A}_{\ideal}$ is called
\emph{universal} if any $\ideal$-exact stable homological functor
$H\colon \mathfrak T\to \mathfrak A$ factors uniquely as
\(\bar{H}\circ U\) for a stable exact functor
$\bar{H}\colon \mathfrak{A}_{\ideal}\to \mathfrak A$.  Such a
universal functor often exists, and then the homological algebra
in~\(\mathfrak{T}\) is very closely related to homological algebra
in the abelian category~$\mathfrak{A}_{\ideal}$.  In particular, the
derived functors in~\(\mathfrak{T}\) are those
in~\(\mathfrak{A}_\ideal\) composed with~\(U\).

Assume~$\mathfrak{T}$ to have countable coproducts.  An object
$C\in\mathfrak T$ is called \emph{$\aleph_1$-compact} if the functor
$\mathfrak{T}(C,{-})\colon \mathfrak T\to \mathfrak{Ab}$ commutes
with countable coproducts.  Let~$\mathfrak{C}$ be an at most
countable set of $\aleph_1$-compact objects in~$\mathfrak{T}$, such
that \(\mathfrak{T}_n(C,A)\defeq\mathfrak{T}(\Sigma^n C, A) \) is
countable for all $A \in \mathfrak{T}$, \(n\in\Z\).  Let
\(\mathfrak{Ab}^{\mathbb{Z}}\) denote the abelian category of
\(\mathbb{Z}\)-graded abelian groups with the suspension
homomorphism shifting degrees.  Define the functor
\[
  F_{\mathfrak{C}}\colon\mathfrak{T} \to
  \prod_{C\in \mathfrak{C}}\mathfrak{Ab}^{\mathbb{Z}}, \qquad
  A \mapsto \bigl(\mathfrak{T}_n(C,A)\bigr)
  _{C \in \mathfrak{C},n\in\mathbb{Z}}.
\]
Let~$\ideal_{\mathfrak{C}}$ be the kernel on morphisms
of~\(F_{\mathfrak{C}}\).  Let
$\langle \mathfrak C\rangle \subseteq \mathfrak{T}$ be the smallest
triangulated subcategory of~$\mathfrak T$ containing~$\mathfrak C$
and closed under countable coproducts.
We are going to describe the universal
\(\ideal_{\mathfrak{C}}\)-exact stable homological functor.
Let~$\mathfrak{C}$ also denote the \(\mathbb{Z}\)\nb-graded
pre-additive category with~\(\mathfrak{C}\) as its object space and
groups of arrows \(\bigoplus_{n\in\Z} \mathfrak{T}_n(A,B)\) for
\(A,B\in\mathfrak{C}\).  A \emph{right \(\mathfrak{C}\)-module} is
defined as a contravariant stable additive functor
\(\mathfrak{C} \to \mathfrak{Ab}^{\mathbb{Z}}\).  These modules form
a stable abelian category with direct sums and enough projective
objects, which we denote by
$\mathfrak{Mod}(\mathfrak{C}^{\mathrm{op}})$.  The subcategory of
countable modules is denoted by
$\mathfrak{Mod}(\mathfrak{C}^{\mathrm{op}})_{\aleph_1}$.  Giving
$\bigl(\mathfrak{T}_n(C,A)\bigr)_{n\in\Z}$ the right
$\mathfrak{C}$\nb-module structure coming from the composition
in~$\mathfrak{T}$, we enrich~$F_{\mathfrak{C}}$ to a functor
\[
  U_{\mathfrak{C}}\colon \mathfrak{T} \rightarrow
  \mathfrak{Mod}(\mathfrak{C}^{\mathrm{op}})_{\aleph_1}.
\]

\begin{lemma}
  \label{lem:universal_functor}
  The universal $\mathfrak{I}_{\mathfrak{C}}$-exact stable
  homological functor is~\(U_\mathfrak{C}\).
\end{lemma}

\begin{proof}
  This is shown during the proof of
  \cite{Meyer-Nest:Filtrated_K}*{Theorem~4.4}.
\end{proof}

\begin{theorem}
  \label{thm:trspectral}
  Let~$\mathfrak T$ be a triangulated category with countable
  coproducts and let $\mathfrak C \subseteq \mathfrak{T}$ be a set
  of $\aleph_1$-compact objects.  Let
  $A\in \langle \mathfrak C\rangle$ and $B \in \mathfrak T$.  Then
  there is a natural, cohomologically indexed, right half-plane,
  conditionally convergent spectral sequence of the form
  \[
    E^{p,q}_2 =
    \Ext^p_{\mathfrak{Mod}(\mathfrak{C}^{\mathrm{op}})}
    \bigl(U_{\mathfrak{C}}(A),U_{\mathfrak{C}}(B)\bigr)_{-q}
    \Rightarrow
    \mathfrak T_{p+q}(A,B).
  \]
  If the object \(U_{\mathfrak{C}}(A)\) has a projective resolution
  of length~$1$, then there is a natural short exact sequence
  \[
    \Ext_{{\mathfrak C}}^1
    (U_{\mathfrak{C}}(\Sigma A),U_{\mathfrak{C}}(B))
    \into \mathfrak{T}(A,B)
    \prto \Hom_{{\mathfrak C}}
    (U_{\mathfrak{C}}(A),U_{\mathfrak{C}}(B)).
  \]
\end{theorem}

\begin{proof}
  These statements are contained in
  \cite{dellAmbrogio:Cell_G}*{Theorem~5.12} and
  \cite{Meyer-Nest:Homology_in_KK}*{Theorem~4.4}.
\end{proof}

\begin{example}
  Let \(\mathfrak{T} = \KK\) and \(\mathfrak{C} = \{\mathbb{C}\}\).
  Then \(\langle \mathfrak C\rangle\) is the well known bootstrap
  class, and the universal functor \(U_{\mathfrak{C}}\) is K-theory,
  viewed as a functor to the stable abelian category
  \(\mathfrak{Ab}^{\mathbb{Z}/2}_{\aleph_1}\) of countable
  \(\mathbb{Z}/2\)-graded abelian groups.  This has global
  homological dimension~$1$.  So the second part of
  Theorem~\ref{thm:trspectral} gives the Universal
  Coefficient Theorem (UCT) of Rosenberg and
  Schochet~\cite{Rosenberg-Schochet:Kunneth}.
\end{example}

\begin{remark}
  The suspension functor $\Sigma = \Cont_0(\mathbb R)\otimes {-}$ in
  $\KK^G$ squares to the identity.  In this situation, the
  \(\mathbb{Z}\)\nb-graded modules in the above discussion become
  \(\mathbb{Z}/2\)-graded.
\end{remark}

The following example is crucial for us.  Fix a finite group~\(G\)
and a conjugation-invariant family~\(\mathcal{F}\) of subgroups of~\(G\).
Let \(\mathfrak{T} = \KK^G\) and let \(\mathfrak{C}\subseteq \KK^G\)
consist of \(\Cont(G/H)\) with the \(G\)\nb-action by translation,
for all subgroups \(H\in\mathcal{F}\).  Since
\(\Cont(G/H) = \Ind_H^G \C\) and \(\Ind_H^G\) is left adjoint to
\(\Res_G^H\), we compute
\[
  \KK^G_*(\Cont(G/H),B)
  \cong \KK^H_*(\C,B)
  \cong \K_*(B\rtimes H).
\]
Consequently, \(\mathfrak{C}\) consists of \(\aleph_1\)\nb-compact
objects.  So Lemma~\ref{lem:universal_functor} applies.  To describe
the universal exact functor in this case, it mostly remains to
understand the arrows in~\(\KK^G\) between the generators
\(\Cont(G/H)\) for \(H\in\mathcal{F}\).  This was done by
Dell'Ambrogio in~\cite{dellAmbrogio:Cell_G}.  He shows that the
family of \(\Z/2\)-graded countable Abelian groups
\(\KK^G_*(\Cont(G/H),B)\) carries the extra structure of a Mackey
module over the representation Green ring~$R^G$ of~\(G\).  We denote
this Mackey module by~\(\mathrm{k}_*^G(B)\)
(see~\cite{Webb:Guide_to_Mackey_functors} for a general, brief
introduction to Mackey and Green functors).  We will do some
computations with Mackey modules in the proofs below and give more
details when they are needed.  It is shown
in~\cite{dellAmbrogio:Cell_G} that the functor~\(\mathrm{k}_*^G\) to
the category \(R^G\textup{-Mac}_{\Z/2,\aleph_1}\) of countable
$\Z/2$-graded Mackey modules over the representation Green
ring~$R^G$ of~\(G\) is the universal homological invariant for the
homological ideal~\(\ideal_{\mathfrak C}\).  In particular, the
following theorem holds:

\begin{theorem}[Dell'Ambrogio~\cite{dellAmbrogio:Cell_G}*{Theorem
    4.9}]
  \label{thm:ivo}
  The restriction of $\mathrm{k}^G\colon \KK^G \to R^G\textup{-Mac}$
  to the full subcategory $\{\Cont(G/H)\mid H\subseteq G\}$ of\/
  $\KK^G$ is fully faithful, that is, for all pairs of subgroups
  \(H,L\subseteq G\) there are canonical isomorphisms
  \[
    \KK^G(\Cont(G/H),\Cont(G/L))\xrightarrow{\mathrm{k}^G}
    R^G\textup{-Mac}(\mathrm{k}^G\Cont(G/H),\mathrm{k}^G\Cont(G/L)).
  \]
\end{theorem}

\section{Generators for the equivariant bootstrap class}
\label{sec:generators_bootstrap}

One way to define the bootstrap class in ordinary \(\KK\)-theory is
as the class of all separable $\Cst$\nb-algebras that are
\(\KK\)-equivalent to a commutative $\Cst$\nb-algebra.  Since all
$\Cst$\nb-algebras of Type~I belong to the bootstrap class, we may
also say that it is the class of all separable $\Cst$\nb-algebras
that are \(\KK\)-equivalent to a Type~I $\Cst$\nb-algebra.  We
choose this definition in the equivariant case.  For any compact
group~$G$, it is shown in
\cite{dellAmbrogio-Emerson-Meyer:Equivariant_Lefschetz}*{Theorem~3.10}
that a separable $G$\nb-$\Cst$\nb-algebras is $\KK^G$-equivalent to
a $G$\nb-action on a Type~I $\Cst$\nb-algebra if and only if it
belongs to the localising subcategory of $\KK^G$ that is generated
by the $G$\nb-actions on ``elementary'' $\Cst$\nb-algebras.  Here a
\(G\)\nb-action on a $\Cst$\nb-algebra is called \emph{elementary}
if it is isomorphic to $\Ind_H^G\mathbb{M}_n(\C)$ for some closed
subgroup $H \subseteq G$ and some group action of~$H$ on the matrix
algebra $\mathbb{M}_n(\C)$ (by automorphisms).  It is shown in the proof
that any \(G\)\nb-action on a $\Cst$\nb-algebra of the form
\(\bigoplus A_n\) where each~\(A_n\) is isomorphic to
\(\Comp(\mathcal{H})\) for a finite-dimensional or separable Hilbert
space~\(\mathcal{H}\) is equivariantly Morita equivalent to a direct
sum of elementary \(G\)\nb-actions.  We also call a
\(\Cst\)\nb-algebra of this form \(\bigoplus A_n\) elementary.  The
Arano--Kubota Theorem~\ref{the:Arano-Kubota} shows that many of the
above generators are redundant:

\begin{theorem}
  \label{the:generate_A_from_induced}
  Let~\(G\) be a finite group.  Then \(A\in\KK^G\) belongs to the
  localising subcategory of~\(\KK^G\) that is generated by
  \(\Cont(G/H) \otimes A\) for cyclic subgroups \(H\subseteq G\).
\end{theorem}

\begin{proof}
  Let \(\ideal \defeq \bigcap_H \ker (\Res_G^H)\), where the
  intersection runs over all cyclic subgroups \(H\subseteq G\).
  Since \(\Res_G^H\) has \(\Ind_H^G\) as a left adjoint functor,
  objects of the form \(\Ind_H^G(A)\) for \(A\in\KK^H\) are
  \(\ideal\)\nb-projective and there are enough
  \(\ideal\)\nb-projective objects in \(\KK^G\) (see
  \cite{Meyer-Nest:Homology_in_KK}*{Proposition~55}).  Since both
  restriction and induction functors commute with direct sums, the
  localising subcategory generated by the induced objects and the
  localising subcategory of \(\ideal\)\nb-contractible objects are a
  complementary pair by
  \cite{Meyer:Homology_in_KK_II}*{Theorem~3.16}.
  Theorem~\ref{the:Arano-Kubota} says that any
  \(\ideal\)\nb-contractible object is already~\(0\).  This means
  that the induced objects generate all of \(\KK^G\).

  Next, we build a specific \(\ideal\)\nb-projective
  resolution of~\(A\).  First, \(\ideal\) is the kernel on morphisms
  of the triangulated functor
  \[
    (\Res_G^H)_{H\ \mathrm{cyclic}}\colon
    \KK^G \to \prod_{H\ \mathrm{cyclic}} \KK^H.
  \]
  This functor has a left adjoint, namely, the functor
  \(\prod_H \KK^H \to \KK^G\),
  \((A_H) \mapsto \bigoplus_H \Ind_H^G (A_H)\).  Then the
  functor
  \[
    T\colon \KK^G \to \KK^G,\qquad
    A\mapsto \bigoplus_{H\ \mathrm{cyclic}} \Ind_H^G \Res_G^H (A)
  \]
  with the counit of the adjunction
  \(\varepsilon \colon T \Rightarrow \mathrm{id}_{\KK^G}\) and the
  comultiplication \(T \Rightarrow T^2\) induced by the unit of the
  adjunction is a comonad in~\(\KK^G\).  Now we can build the bar
  resolution of~\(A\) with the objects \(T^{n+1}(A)\) and the
  boundary map
  \(\sum_{j=1}^{n+1} (-1)^j \varepsilon_j\colon T^n(A) \to
  T^{n-1}(A)\), where~\(\varepsilon_j\) is the whiskering of
  \(\varepsilon\colon T\Rightarrow \mathrm{id}_{\KK^G}\)
  by~\(T^{j-1}\) on the left and~\(T^{n-j}\) on the right
  (see~\cite{Barr-Beck:Homology} for the construction and properties
  of the bar resolution in this generality).

  Since the objects of the form \(T(A)\) are all
  \(\ideal\)\nb-projective, the bar resolution above is an
  \(\ideal\)\nb-projective resolution of~\(A\).  Next, we build a
  ``phantom castle'' from this \(\ideal\)\nb-projective resolution
  as in \cite{Meyer:Homology_in_KK_II}*{Section~3}.  This contains
  \(\ideal\)\nb-cellular approximations of~\(A\), and their homotopy
  colimit is isomorphic to~\(A\) by
  \cite{Meyer:Homology_in_KK_II}*{Proposition~3.18} because all
  \(\ideal\)\nb-contractible objects are~\(0\).  It follows
  that~\(A\) belongs to the localising subcategory of \(\KK^G\) that
  is generated by \(T^k(A)\) for \(k\ge1\).

  By construction, \(T^k(A)\) is the direct sum of the tensor
  products
  \[
    \Cont(G/H_1) \otimes \Cont(G/H_2) \otimes \dotsb \otimes
    \Cont(G/H_k) \otimes A
    \cong
    \Cont(G/H_1 \times G/H_2 \times \dotsb \times G/H_k, A)
  \]
  for cyclic subgroups \(H_1,\dotsc,H_k\subseteq G\).  Decomposing
  \(G/H_1 \times \dotsb \times G/H_k\) into orbits, we further
  decompose this as a direct sum of \(\Cont(G/H,A)\) where
  \(H\subseteq G\) is the stabiliser of an orbit representative.
  Each such stabiliser will be contained in a group that is
  conjugate to~\(H_1\), making it cyclic as well.  Therefore,
  \(T^k(A)\) is isomorphic to a direct sum of \(\Cont(G/H,A)\) for
  cyclic subgroups \(H\subseteq G\).
\end{proof}

\begin{corollary}
  An object~\(A\) in \(\KK^G\) belongs to the equivariant bootstrap
  class if and only if \(\Res_G^H(A)\) belongs to the equivariant
  bootstrap class in \(\KK^H\) for each cyclic subgroup
  \(H\subseteq G\).
\end{corollary}

\begin{proof}
  Both restriction and induction functors map actions on Type~I
  \(\Cst\)\nb-algebras again to actions on Type~I
  \(\Cst\)\nb-algebras.  Therefore, they map the equivariant
  bootstrap classes to each other.  If \(\Res_G^H A\) is in the
  \(H\)\nb-equivariant bootstrap class, so is
  \(\Cont(G/H, A) \cong \Ind_H^G \Res_G^H A\).  Now
  Theorem~\ref{the:generate_A_from_induced} implies the result.
\end{proof}

\begin{corollary}
  \label{cor:generate_bootstrap}
  The objects \(\Cont(G/H)\) for cyclic subgroups \(H\subseteq G\)
  generate the equivariant bootstrap class in \(\KK^G\).
\end{corollary}

\begin{proof}
  It suffices to prove that the localising subcategory generated by
  \(\Cont(G/H)\) for cyclic subgroups \(H\subseteq G\) contains all
  the generators of the equivariant bootstrap class.  We therefore
  pick one or, a bit more generally, a direct sum of these
  generators.  So let~\(A\) be a \(G\)\nb-action on an elementary
  \(\Cst\)\nb-algebra.  By
  Theorem~\ref{the:generate_A_from_induced}, \(A\) belongs to the
  localising subcategory generated by
  \(\Cont(G/H,A) \cong \Ind_H^G \Res_G^H A\) for cyclic subgroups
  \(H\subseteq G\).  Thus \(G\)\nb-\(\Cst\)-algebras of the form
  \(\Ind_H^G B\) for cyclic subgroups \(H\subseteq G\) and an action
  of~\(H\) on an elementary \(\Cst\)\nb-algebra~\(B\) also generate
  the equivariant bootstrap class; they cannot generate a larger
  subcategory because \(\Ind_H^G B\) is an elementary
  \(\Cst\)\nb-algebra if~\(B\) is.  Now for a cyclic group~\(H\),
  any \(2\)\nb-cocycle is trivial, so any elementary
  \(H\)\nb-\(\Cst\)-algebra is Morita equivalent to \(\Cont(H/K)\)
  for a subgroup \(K\subseteq H\), which is again cyclic.
  Thus~\(A\) belongs to the localising subcategory generated by
  \(\Ind_H^G \Cont(H/K) \cong \Cont(G/K)\) for cyclic subgroups
  \(K\subseteq G\).
\end{proof}

The next corollary removes the finite generation assumption from
\cite{Arano-Kubota:Atiyah-Segal}*{Corollary~3.23.(1)}.

\begin{corollary}
  Let \(A\) and~\(B\) be objects of \(\KK^G\).  If
  \(\KK^H_*(A,B)=0\) for all cyclic subgroups \(H\subseteq G\), then
  \(\KK^G_*(A,B)=0\).
\end{corollary}

\begin{proof}
  Since induction is left adjoint to restriction, the assumption is
  equivalent to \(\KK^G_*(\Cont(G/H) \otimes A,B)=0\) for all cyclic
  subgroups \(H\subseteq G\).  The class of objects~\(D\) with
  \(\KK^G_*(D,B)=0\) is localising.  So the claim follows from
  Theorem~\ref{the:generate_A_from_induced}.
\end{proof}

The following corollary relates certain conditions that are clearly
necessary for \(A\in\KK^G\) to belong to the equivariant bootstrap
class.  We do not know whether they are also sufficient.

\begin{corollary}
  Let~\(A\) be an object of \(\KK^G\).  If \(A\rtimes H\) belongs to
  the bootstrap class in \(\KK\) for all cyclic subgroups
  \(H\subseteq G\), then \(A\rtimes K\) is in the bootstrap class in
  \(\KK\) for all subgroups, cyclic or not.  In addition,
  \((A\otimes B)\rtimes G\) is in the bootstrap class in \(\KK\)
  if~\(B\) is in the bootstrap class in \(\KK^G\).
\end{corollary}

\begin{proof}
  Since \(\bigl(A\otimes \Cont(G/H)\bigr)\rtimes G\) is
  Morita--Rieffel equivalent to \(A\rtimes H\), the assumption means
  that \(\bigl(A\otimes \Cont(G/H)\bigr)\rtimes G\) belongs to the
  bootstrap class.  Since tensoring with~\(A\) and the crossed
  product with~\(G\) are triangulated functors that commute with
  countable direct sums, this implies that
  \(\bigl(A\otimes B\bigr)\rtimes G\) belongs to the bootstrap class
  for all~\(B\) in the localising subcategory generated by
  \(\Cont(G/H)\) for the cyclic subgroups \(H\subseteq G\).  This is
  the equivariant bootstrap class by
  Corollary~\ref{cor:generate_bootstrap}.  Since it contains
  \(\Cont(G/K)\) for any subgroup \(K\subseteq G\), we also get the
  claim about \(A\rtimes K\), which is Morita--Rieffel equivalent to
  \(\bigl(A\otimes \Cont(G/K)\bigr)\rtimes G\).
\end{proof}

Let~$B$ be a separable $G$\nb-$\Cst$\nb-algebra in the equivariant
bootstrap class.  Using the Ind-Res adjunction, for all \(H\subseteq G\),
\[
  \KK^G_*(\Cont(G/H),B)
  \cong \KK^G_*(\Ind_H^G \C,B)
  \cong \KK^H_*(\C,\Res^G_H B)
  \cong \K^H_*(B)
\]
is a countable \(\Z/2\)-graded module over the representation ring
of~\(H\).  The representation rings of all subgroups of~\(G\) form a
Green functor~$R^G$ (see~\cite{dellAmbrogio:Cell_G}).  The set of cyclic
subgroups of $G$ is closed under taking subgroups and conjugation. This allows
to consider the representation Green functor only on cyclic subgroups of~$G$. We denote it by~$R^G_{\textup{cy}}$.  Then the countable \(\Z/2\)-graded \(R(H)\)-modules \(\K^H_*(B)\) for the cyclic
\(H\subseteq G\) form a Mackey functor on cyclic subgroups of~\(G\) over~$R^G_{\textup{cy}}$.

\begin{proposition}[\cite{dellAmbrogio:Cell_G}]
  \label{pro:Mackey_universal}
  The representable functor
  \[
    \mathrm{ck}_*^G\colon \KK^G
    \to \mathfrak{Mod}^{\Z/2}(R^G_{\textup{cy}})_{\aleph_1},\qquad
    B\mapsto \bigl\{\K_*^H
    \bigl(B)\bigr\}
    ^{\textup{cyclic }H\subseteq G},
  \]
  into the abelian category of $\Z/2$\nb-graded countable right
  Mackey modules on cyclic subgroups of~\(G\) over~$R^G_{\textup{cy}}$ is the universal stable $\ker \mathrm{ck}_*^G$\nb-exact functor.
\end{proposition}

\begin{corollary}
  \label{cor:KKspectral}
  Let~$G$ be a finite group.  For every $A, B \in \KK^G$ with~$A$ in
  the \(G\)\nb-equivariant boostrap class, there is a
  cohomologically indexed, right half plane, conditionally
  convergent spectral sequence
  \[
    E^{p,q}_2=\Ext^p_{R^G_{\textup{cy}}}
    \bigl(\mathrm{ck}_*^G(A),\mathrm{ck}_*^G(B)\bigr)_{-q}
    \xRightarrow{n=p+q}\KK^G_n(A,B)
  \]
  that depends functorially on $A$ and~$B$.
\end{corollary}

\begin{proof}
  This follows from Corollary~\ref{cor:generate_bootstrap},
  Theorem~\ref{thm:trspectral}, and
  Proposition~\ref{pro:Mackey_universal}.
\end{proof}

\section{Localisation}

\subsection{Localisation at a set of primes}
\label{sec:localisation_primes}

We recall how to localise~\(\KK^G\) at a set of primes~\(S\); this
works also if~\(G\) is an arbitrary locally compact group, or for
other types of equivariant \(\KK\)-theory.  Let
\(\Z[S^{-1}]\defeq \Z[1/p, p\in S]\).  There are two useful ways to
localise the category \(\KK^G\) by \(\Z[S^{-1}]\).  We may either
take the arrows between \(A\) and~\(B\) to be
\(\KK^G(A,B) \otimes_\Z \Z[S^{-1}]\) as
in~\cite{Inassaridze-Kandelaki-Meyer:Finite_Torsion_KK} or
\(\KK^G(A, B \otimes \Mat_{S^\infty})\) as in
\cite{Blackadar:K-theory}*{Exercise~23.15.6};
here~\(\Mat_{S^\infty}\) denotes the UHF algebra of type
\(\prod_{p\in S} p^\infty\) with the trivial action of~\(G\).  The
first localisation yields again a triangulated category, but the
canonical functor from \(\KK^G\) to it does not preserve coproducts.
Therefore, our machinery of relative homological algebra applies
only partially.  This is why we prefer the second approach to
localisation here.

\begin{definition}
  \label{def:S-divisible}
  A separable \(G\)-\(\Cst\)-algebra~\(A\) is
  \emph{\(S\)\nb-divisible} if
  \(p\cdot \mathrm{id}_A \in \KK^G(A,A)\) is invertible for all
  \(p\in S\).
\end{definition}

\begin{remark}
  If~\(A\) is \(S\)\nb-divisible, then for each \(p\in S\) there is
  \(h\in \KK^G(A,A)\) with \(p\cdot h = \mathrm{id}_A\).  The
  converse is also true because the Kasparov product with
  \(p\cdot \mathrm{id}_A\in\KK^G(A,A)\) on either side simply
  multiplies by~\(p\).
\end{remark}

\begin{proposition}
  \label{pro:S-divisible}
  A separable \(G\)-\(\Cst\)-algebra~\(A\) is \(S\)\nb-divisible if
  and only if the canonical map \(A\to A \otimes \Mat_{S^\infty}\)
  is a \(\KK^G\)-equivalence, if and only if~\(A\) is isomorphic to
  \(B \otimes \Mat_{S^\infty}\) for some separable
  \(G\)-\(\Cst\)-algebra~\(B\).  If~\(B\) is \(S\)\nb-divisible,
  then
  \begin{equation}
    \label{eq:Mat_S_isomorphism}
    \KK^G(A,B) \cong \KK^G(A \otimes \Mat_{S^\infty},B).
  \end{equation}
\end{proposition}

\begin{proof}
  Let~\(B\) be \(S\)\nb-divisible.  We are going to prove that the
  canonical inclusion
  \(A \hookrightarrow A \otimes \Mat_{S^\infty}\) induces an
  isomorphism as in~\eqref{eq:Mat_S_isomorphism}.  The
  \(\Cst\)\nb-algebra~\(\Mat_{S^\infty}\) is defined as the
  \(\Cst\)\nb-algebraic inductive limit of an inductive system
  formed of maps \(\Mat_{m_n}(\C) \to \Mat_{m_{n+1}}(\C)\) for
  \(m_n\in\N^\times\) with \(m_0=1\) and \(m_{n+1} = p_n\cdot m_n\),
  such that \(p_n\in S\) for all \(n\in\N\) and each element
  of~\(S\) occurs infinitely many times among the~\(p_n\).
  Since~\(\Mat_{S^\infty}\) is nuclear, the inductive system
  considered is ``admissible'' (see~\cite{Meyer-Nest:BC}).
  So~\(\Mat_{S^\infty}\) is a homotopy colimit as well, that is,
  there is an exact triangle
  \[
    \bigoplus \Mat_{m_n}(\C) \xrightarrow{\mathrm{id} - \sigma}
    \bigoplus \Mat_{m_n}(\C) \to
    \Mat_{S^\infty} \to \Sigma \bigoplus \Mat_{m_n}(\C),
  \]
  where~\(\sigma\) is the map induced by the inclusions
  \(\Mat_{m_n}(\C) \to \Mat_{m_{n+1}}(\C)\).  Since the tensor
  product with~\(A\) is a triangulated functor,
  \(A\otimes \Mat_{S^\infty}\) is the homotopy colimit of the
  induced inductive system
  \begin{equation}
    \label{eq:inductive_system_Mat_S}
    A = \Mat_{m_0}(A) \to
    \Mat_{m_1}(A) \to \dotsb \to
    \Mat_{m_n}(A) \to \Mat_{m_{n+1}}(A)
    \to \dotsb.
  \end{equation}
  When we compose the induced map
  \(\KK^G(\Mat_{m_{n+1}}(A),B)\to \KK^G(\Mat_{m_n}(A),B)\) with the
  canonical Morita equivalences between~\(A\) and \(\Mat_{p_n}(A)\),
  it becomes multiplication by~\(p_n\) on
  \(\KK^G(A,B)\).  Since~\(B\) is assumed \(S\)\nb-divisible, this
  is invertible for all~\(n\).  Therefore, the long exact sequence
  for a homotopy colimit simplifies to show that the inclusions
  \(\Mat_{m_n}(A) \hookrightarrow A \otimes \Mat_{S^\infty}\) induce
  isomorphisms on \(\KK^G({-},B)\) for all \(n\in\N\).  For \(n=0\),
  this becomes the isomorphism in~\eqref{eq:Mat_S_isomorphism}.

  Next, assume that~\(A\) is \(S\)\nb-divisible.  Then the maps
  in~\eqref{eq:inductive_system_Mat_S} are \(\KK^G\)-equivalences.
  It follows that the homotopy colimit of this inductive system is
  \(\KK^G\)-equivalent to \(\Mat_{m_n}(A)\) for all \(n\in\N\).  For
  \(n=0\), this gives the desired \(\KK^G\)-equivalence
  between~\(A\) and \(A \otimes \Mat_{S^\infty}\).  Conversely,
  since \(p\cdot \mathrm{id}_{\Mat_{S^\infty}}\) is invertible in
  \(\KK(\Mat_{S^\infty}, \Mat_{S^\infty})\) for all \(p\in S\),
  anything \(\KK^G\)-equivalent to \(A \otimes \Mat_{S^\infty}\) for
  some~\(A\) is \(S\)\nb-divisible.
\end{proof}

\begin{lemma}
  \label{lem:S-divisible_localising}
  The \(S\)\nb-divisible objects in \(\KK^G\) form a localising
  subcategory, that is, it is thick and closed under countable
  coproducts.  We denote it by~\(\KK^G_S\).
\end{lemma}

\begin{proof}
  The functoriality of suspensions and coproducts shows that a
  countable coproduct of suspensions of \(S\)\nb-divisible objects is
  again \(S\)\nb-divisible.  Let \(A \to B \to C \to \Sigma A\) be an
  exact triangle in \(\KK^G\).  Multiplication by~\(p\) is a
  \(\KK^G\)-equivalence if and only if its mapping cone is
  \(\KK^G\)-equivalent to~\(0\).  Since multiplication by~\(p\) is
  natural, the cones of multiplication by~\(p\) on \(A\), \(B\)
  and~\(C\) also form an exact triangle by
  \cite{Beilinson-Bernstein-Deligne}*{Proposition~1.1.11}.  Therefore,
  if two of \(A\), \(B\) and~\(C\) are \(p\)\nb-divisible for a
  prime~\(p\), so is the third.  Thus the class of \(S\)\nb-divisible
  objects in \(\KK^G\) is triangulated.

  It is
  known that a triangulated category with at least countable direct
  sums is Karoubian, that is, any idempotent has an image object; in
  particular, a triangulated subcategory closed under direct sums
  --~such as that of \(S\)\nb-divisible objects~-- is closed under
  direct summands, making it thick
  (see~\cite{Neeman:Triangulated}*{Remark 3.2.7}).  We recall how
  this is shown.  A direct summand of an object~\(A\) is the image
  of an idempotent endomorphism \(p\colon A \to A\).  The homotopy
  colimit of the constant inductive system
  \begin{equation}
    \label{eq:constant_inductive_system}
    A \xrightarrow{p} A \xrightarrow{p} A \xrightarrow{p} A \to \dotsb
  \end{equation}
  exists and has the universal property of an image object
  for~\(p\).
\end{proof}

\begin{proposition}
  \label{pro:KKGS_two_ways}
  The localising subcategory \(\KK^G_S\subseteq \KK^G\) is
  equivalent to the category with the same objects as~\(\KK^G\) and
  \(\KK^G(A, B\otimes \Mat_{S^\infty})\) as the arrows from~\(A\)
  to~\(B\), and with the composition induced by the Kasparov product
  in \(\KK^G\) followed by the canonical \(\KK^G\)-equivalence
  \(\Mat_{S^\infty} \otimes \Mat_{S^\infty} \cong \Mat_{S^\infty}\).
\end{proposition}

\begin{proof}
  It is straightforward to check that there is a category with the
  same objects as \(\KK^G\), with the arrows
  \(\KK^G(A, B\otimes \Mat_{S^\infty})\), and with the
  multiplication specified in the statement.
  Proposition~\ref{pro:S-divisible} shows first that, in this
  category, every object~\(A\) is isomorphic to
  \(A\otimes \Mat_{S^\infty}\), secondly, that the latter is
  \(S\)\nb-divisible, and, thirdly, that among \(S\)\nb-divisible
  objects, arrows in this category simplify to \(\KK^G(A, B)\) with
  the usual Kasparov product as composite.
\end{proof}

We now apply the machinery of relative homological algebra to the
class of objects \(\Cont(G/H) \otimes \Mat_{S^\infty}\) for cyclic $H\subseteq G$
in the category~\(\KK^G_S\).  Whereas these objects are not
$\aleph_1$\nb-compact in \(\KK^G\), they are $\aleph_1$\nb-compact
in~\(\KK^G_S\) because
of~\eqref{eq:Mat_S_isomorphism}.

\begin{proposition}
  \label{pro:localised_bootstrap}
  The localising subcategory of~\(\KK^G_S\) generated by the objects
  \(\Cont(G/H) \otimes \Mat_{S^\infty}\) for cyclic $H\subseteq G$
  consists precisely of the \(S\)\nb-divisible objects in the equivariant bootstrap class
  in~\(\KK^G\).  An object~\(B\) in the equivariant bootstrap class
  in~\(\KK^G\) is \(S\)\nb-divisible if and only if multiplication
  by~\(p\) is an isomorphism on \(\KK^G(\Cont(G/H),B)\) for all cyclic \(H \subseteq   G\) and all \(p\in S\).
\end{proposition}

\begin{proof}
  As a homotopy colimit of objects of the form \(\Mat_m(\Cont(G/H))\),
  the generators \(\Cont(G/H) \otimes \Mat_{S^\infty}\) belong to the
  equivariant bootstrap class.  Hence the localising subcategory
  generated by them is contained in the latter.  It consists of
  \(S\)\nb-divisible objects by
  Lemma~\ref{lem:S-divisible_localising}.

  Conversely, let~\(B\) be \(S\)\nb-divisible and in the equivariant
  bootstrap class.  We do homological algebra in~\(\KK^G\) using the
  objects \(\{\Cont(G/H)\mid H\subseteq G \text{ cyclic}\}\); these
  are the generators of the bootstrap class by
  Corollary~\ref{cor:generate_bootstrap}.  There are enough relative
  projective objects and we may build a cellular approximation tower
  for~\(B\) from a projective resolution~\((P_n)\) as
  in~\cite{Meyer:Homology_in_KK_II}.  This is a sequence of exact
  triangles
  \[
    \tilde{B}_n\to\tilde{B}_{n+1}\to P_n\to \Sigma \tilde{B}_n.
  \]
  with certain properties.  It follows that \(\tilde{B}_n\) and~$P_n$
  are in the equivariant bootstrap class.  The homotopy colimit
  of~\((\tilde{B}_n)\) is~\(B\) by
  \cite{Meyer:Homology_in_KK_II}*{Proposition~3.18}.  Now, the tensor
  product of our cellular approximation tower with~\(\Mat_{S^\infty}\),
  \[
    \tilde{B}_n\otimes \Mat_{S^\infty}
    \to\tilde{B}_{n+1}\otimes \Mat_{S^\infty}
    \to P_n\otimes \Mat_{S^\infty}
    \to \Sigma \tilde{B}_n\otimes \Mat_{S^\infty},
  \]
  gives a cellular approximation tower in~\(\KK^G_S\).  Its homotopy
  colimit is \(B\otimes \Mat_{S^\infty}\).  Therefore, the latter
  belongs to the localising subcategory generated by the objects
  \(\Cont(G/H) \otimes \Mat_{S^\infty}\) with cyclic~$H$.  Since~\(B\)
  is \(S\)\nb-divisible, \(B \cong B\otimes \Mat_{S^\infty}\).

  Finally, we prove the criterion for \(S\)\nb-divisibility.  Since
  each~\(\Cont(G/H)\) is \(\aleph_1\)\nb-compact,
  \[
    \KK^G_*(\Cont(G/H),B \otimes \Mat_{S^\infty})
    \cong \KK^G_*(\Cont(G/H),B) \otimes \Z[S^{-1}]
  \]
  holds for all~\(B\).  (We remark that this isomorphism relates the
  approach to localisation that we follow here to that by tensoring
  the arrow spaces with \(\Z[S^{-1}]\).)  Therefore,
  \(\KK^G_*(\Cont(G/H),B) \cong \KK^G_*(\Cont(G/H),B) \otimes
  \Z[S^{-1}]\) holds if~\(B\) is \(S\)\nb-divisible, and
  multiplication by~\(p\) for \(p\in S\) is invertible on this
  group.  Conversely, assume multiplication by~\(p\) for \(p\in S\)
  is invertible on \(\KK^G_*(\Cont(G/H),B)\) for all cyclic
  \(H\subseteq G\).  We must show that \(p\cdot \mathrm{id}_B\) is
  invertible in \(\KK^G_0(B,B)\).  Equivalently, its mapping
  cone~\(C\) is~\(0\).  Since this mapping cone still belongs to the
  equivariant bootstrap class, \(C \cong 0\) if and only
  \(\KK^G_*(\Cont(G/H),C) \cong 0\) for all cyclic \(H\subseteq G\).
  By the Puppe sequence, this happens if and only if multiplication
  by~\(p\) is an isomorphism on \(\KK^G_*(\Cont(G/H),B)\).
\end{proof}

\section{Localisation at the group order}
\label{sec:localisation_group_order}

While the modular representation theory of groups may be very
complicated, it becomes relatively easy over a field in which the
group order is invertible.  In this section, we simplify the
equivariant bootstrap class after localising at the group order.
For finite cyclic groups, this is already shown by Manuel K\"ohler
(see~\cite{Koehler:Thesis}*{Theorem~13.1}).

Let~\(S\) be the (finite) set of
primes that divide the order~\(\abs{G}\) of~\(G\).  We are going to
work in the localising subcategory of \(S\)\nb-divisible objects in
the equivariant bootstrap class in \(\KK^G\).  This is described in
Proposition~\ref{pro:localised_bootstrap} as the localising
subcategory generated by the objects
\(\Cont(G/H) \otimes \Mat_{S^\infty}\) for cyclic subgroups \(H\subseteq G\).

\begin{remark}
  We could also localise at a larger set of primes or even tensor
  with~\(\mathbb{Q}\) as
  in~\cite{Bouc-DellAmbrogio-Martos:Splitting}.  We do not discuss
  this because such a localisation may be obtained by first
  localising at the primes dividing~\(\abs{G}\) and then localising
  once again at the remaining primes.  So the statements we are
  going to prove imply the more general statements.
\end{remark}

The main result in this section is a Universal Coefficient Theorem
in this setting.  We first formulate this theorem, which requires
some notation.  For a cyclic subgroup \(H\subseteq G\), let
\(n\defeq \abs{H}\), let~\(\vartheta_n\) be a primitive \(n\)th root
of unity, and let
\[
  N_H \defeq \{y\in G \mid y H y^{-1} = H\}, \qquad
  W_H \defeq N_H/H.
\]
The representation ring of~\(H\) is isomorphic to~\(\Z[z]/(z^n-1)\),
and \(\Z[\vartheta_n] \subseteq \C\) is a quotient of that by
evaluation at~\(\vartheta_n\).  The \(n\)th cyclotomic
polynomial~\(\Phi_n\) is the minimal polynomial of~\(\vartheta_n\),
that is, \(\Z[\vartheta_n] \cong \Z[z]/(\Phi_n)\).  The quotient
\(\Z[\vartheta_n]\) of the representation ring of~\(H\) is invariant
under the induced action of group automorphisms of~\(H\).
Conjugation by elements of~\(N_H\) defines automorphisms of~\(H\),
so that~\(N_H\) acts on \(\Z[\vartheta_n]\) in a canonical way.
Since elements of~\(H\) act trivially, this induces an action
of~\(W_H\) on the representation ring and then on the quotient
\(\Z[\vartheta_n]\).  Let \(\Z[\vartheta_n] \rtimes W_H\) be the
resulting semidirect product.

Let~\(B\) be an object of \(\KK^G_S\), that is, \(B\) is a
\(G\)\nb-action on a separable \(\Cst\)\nb-algebra and \(B\) is
\(\KK^G\)-equivalent to \(\Mat_{S^\infty} \otimes B\).  The elements
corresponding to $H$\nb-representations and conjugations span a
subring in the endomorphism ring \(\KK^G(\Cont(G/H),\Cont(G/H))\)
that is isomorphic to \(\Z[z]/(z^n-1) \rtimes W_H\).  We quickly
explained this in the introduction, and it also follows from
Theorem~\ref{thm:ivo}.  So \(\KK^G_*(\Cont(G/H),B)\) becomes a
\(\Z/2\)-graded module over the latter ring.  Let
\[
  F_*^H(B) \defeq
  \{x \in \KK^G_*(\Cont(G/H),B) \mid \Phi_n(z)\cdot x =0\}.
\]
This subgroup is a \(\Z/2\)-graded module over the quotient ring
\(\Z[\vartheta_n] \rtimes W_H\).  Since~\(B\) is \(S\)\nb-divisible,
multiplication by~\(\abs{G}\) is invertible on~\(F_*^H(B)\), so that it
becomes a \(\Z/2\)-graded module over
\(\Z[\vartheta_n,1/\abs{G}] \rtimes W_H\).

\begin{theorem}
  \label{the:UCT_S}
  Let~\(G\) be a finite group.  Let~\(\mathfrak{A}_G\) be the
  product of the categories of\/ \(\Z/2\)-graded, countable modules
  over the rings \(\Z[\vartheta_n,1/\abs{G}] \rtimes W_H\),
  where~\(H\) runs through a set of representatives for the
  conjugacy classes of cyclic subgroups in~\(G\).  This stable
  Abelian category is hereditary, that is, any object has a
  projective resolution of length~\(1\).  The functors~\(F_*^H\) for
  these~\(H\) combine to a stable homological functor
  \(F\colon \KK^G_S \to \mathfrak{A}_G\).  If \(A,B\in\KK^G_S\)
  and~\(A\) belongs to the equivariant bootstrap class in \(\KK^G\),
  then there is a Universal Coefficient Theorem
  \[
    \Ext_{\mathfrak{A}_G} \bigl( F(A), F(\Sigma B)\bigr)
    \into \KK^G_S(A,B)
    \prto\Hom_{\mathfrak{A}_G} \bigl( F(A), F(B)\bigr).
  \]

  The functor~\(F\) induces a bijection between isomorphism classes
  of \(S\)\nb-divisible objects in the \(G\)\nb-equivariant
  bootstrap class and isomorphism classes of objects
  in~\(\mathfrak{A}_G\).
\end{theorem}

We will prove this theorem in the remainder of this section.

We begin by defining an idempotent in~\(\Cont(G/H)\) that produces
the invariant~\(F_*^H\).  It involves fractions with~\(\abs{G}\) in
the denominator, so that it only exists after inverting the primes
in~\(S\).  The construction uses some facts about representation
rings and cyclotomic polynomials which are already used by Köhler
in~\cite{Koehler:Thesis} to prove the special case of our main
result when the whole group~\(G\) is cyclic.

Let \(\Phi_k(z)\in \Z[z]\) be the \(k\)th cyclotomic polynomial, whose
roots are exactly the primitive \(k\)th roots of unity.  Then
\begin{equation}
  \label{eq:factorize_cyclotomic}
  z^n-1 =  \prod_{k\mid n} \Phi_k(z).
\end{equation}
For \(k \mid n\), let
\[
  \psi_{n,k}(z) \defeq \frac{z}{n} \frac{\mathrm{d} \Phi_k(z)}{\mathrm{d}z}
    \cdot \prod_{k'\mid n, k' \neq k} \Phi_{k'}(z)
    = \frac{z(z^n-1)}{n \Phi_k(z)} \frac{\mathrm{d}\Phi_k(z)}{\mathrm{d}z}.
\]
By definition, \(n\cdot \psi_{n,k} \in \Z[z]\).  So~\(\psi_{n,k}\)
only becomes available after localisation at~\(n\).

\begin{lemma}[\cite{Koehler:Thesis}*{Lemmas 22.5 and 22.6}]
  \label{lem:psi_lelations_Koehler}
  The polynomials \(\psi_{n,k}\) form a complementary set of
  idempotent elements in the ring \(\Z[z,1/n]/(z^n-1)\), that is,
  \[
    \psi_{n,k}\cdot \psi_{n,l}
    \equiv \delta_{k,l} \psi_{n,k} \bmod (z^n-1),\qquad
    \sum_{k\mid n} \psi_{n,k} \equiv 1 \bmod (z^n-1).
  \]
\end{lemma}

\begin{proof}
  The relation \(\sum_{k\mid n} \psi_{n,k} \equiv 1 \bmod (z^n-1)\)
  follows by differentiating \(z^n-1 = \prod_{k\mid n} \Phi_k(z)\)
  and multiplying by~\(z/n\).  The
  relation~\eqref{eq:factorize_cyclotomic} implies that \(z^n-1\)
  divides \(\psi_{n,k}\psi_{n,l}\) if \(k\neq l\), giving
  \(\psi_{n,k}\cdot \psi_{n,l} \equiv 0 \bmod (z^n-1)\) in this
  case.  Together with
  \(\sum_{k\mid n} \psi_{n,k} \equiv 1 \bmod (z^n-1)\), this implies
  \(\psi_{n,k}\cdot \psi_{n,k} \equiv \psi_{n,k} \bmod (z^n-1)\).
\end{proof}

The lemma implies an isomorphism of rings
\begin{equation}
  \label{eq:cyclic_rep_localized}
  \Z[z,1/n]/(z^n-1)
  \cong \bigoplus_{k\mid n} \Z[z,1/n]/(\Phi_k)
  \cong \bigoplus_{k\mid n} \Z[\vartheta_k,1/n],
\end{equation}
where \(\vartheta_k\in \C\) is a primitive \(k\)th root of unity
(see \cite{Koehler:Thesis}*{Proposition 22.8}).

Next we are going to compute the character of~\(\psi_{n,k}\).
Mapping a representation to its character defines an injective map
from the representation ring to the ring of class functions with
pointwise multiplication.  This remains injective after inverting
some rational numbers.  Since our cyclic group~\(H\) is Abelian, all
functions are class functions.  The character homomorphism maps the
generator \(z\in \Z[z]/(z^n-1)\) of the representation ring to the
function \(\Z/n\to \C\), \(j\mapsto \vartheta_n^j\).  Thus the image
of \(p\in \Z[z,1/n]/(z^n-1)\) is the function that maps
\(j\in \Z/n\) to \(p(\vartheta_n^j)\).

\begin{lemma}
  \label{lem:character_psi}
  The character of~\(\psi_{n,k}\) is the characteristic function of
  the subset of elements of~\(\Z/n\) of order equal to~\(k\).
\end{lemma}

\begin{proof}
  By construction, \(\psi_{n,k}\) vanishes at the primitive \(l\)th
  roots of unity for all divisors \(l \mid n\) with \(l\neq k\).
  Then Lemma~\ref{lem:psi_lelations_Koehler} shows that the
  character of~\(\psi_{n,k}\) is the characteristic function of the
  set of all \(j\in \Z/n\) for which~\(\vartheta_n^j\) is a
  primitive \(k\)th root of unity.  This is equivalent to~\(j\)
  having order equal to~\(k\) in~\(\Z/n\).
\end{proof}

The endomorphism ring \(\KK^G(\Cont(G/H),\Cont(G/H))\) in
Theorem~\ref{thm:ivo} was computed by K\"ohler~\cite{Koehler:Thesis}
for cyclic groups and in general by Ivo
Dell'Ambrogio~\cite{dellAmbrogio:Cell_G}.  We have already explained
in the introduction how to map the representation ring of~\(H\) into
it.  When~\(H\) is a cyclic subgroup of order~\(n\), we thus get an
embedding
\[
  \Z[z]/(z^n-1) \hookrightarrow\KK^G(\Cont(G/H),\Cont(G/H)).
\]
In particular, the integer polynomials~\(n\psi_{n,k}\) give elements
in this ring.  After localising at~\(\abs{G}\), we may divide these
elements by~\(n\) and get idempotent elements
\[
  p_{n,k} \in \KK^G_S(\Cont(G/H) \otimes \Mat_{S^\infty},\Cont(G/H)
  \otimes \Mat_{S^\infty}).
\]
These satisfy the relations in
Lemma~\ref{lem:psi_lelations_Koehler}.  In particular, they are
complementary idempotents.  Actually, we only need the idempotent
element \(p_n \defeq p_{n,n}\).  The proof of
Lemma~\ref{lem:S-divisible_localising} shows that it has an image
object in \(\KK^G_S\), which we denote by~\(A_H^0\).

\begin{proposition}
  \label{pro:fewer_generators_localisation}
  Let \(H\subseteq G\) be a cyclic subgroup.  Assume that~\(S\)
  contains all prime divisors of~\(\abs{H}\).  In the localisation
  of~\(\KK^G\) at~\(S\), the object \(\Cont(G/H)\) in~\(\KK^G\)
  becomes isomorphic to a direct sum of~\(A_H^0\) and certain direct
  summands of~\(A_K^0\) for subgroups \(K\subseteq H\).
\end{proposition}

\begin{proof}
  Recall that any idempotent in \(\KK^G\) has an image object.
  Therefore, we may write~\(A_H\) as the direct sum of the image
  objects of the complementary orthogonal idempotents~\(p_{n,k}\)
  for \(k \mid n\).  By definition, \(A_H^0\) is an image object
  for~\(p_{n,n}\).  We finish the proof of the proposition by
  showing that the image object of~\(p_{n,k}\) for a proper
  divisor~\(k\) of~\(n\) is isomorphic to a direct summand of~\(A_K^0\), where
  \(K\subseteq H\) is the cyclic subgroup with \(k\)~elements.  Here
  we use the Frobenius relation for the induction and restriction
  generators for the subgroup \(K\subseteq H\) and the idempotent
  element~\(p_{k,k}\) that projects~\(A_K\) onto~\(A_K^0\)
  (see~\cite{dellAmbrogio:Cell_G}*{Section~3.1}); this applies here
  because the groups \(\KK^G_*(\Cont(G/H),{-})\) for
  \(H\subseteq G\) form a Mackey module over the Green functor of
  representation rings, tensored
  by~\(\Z[S^{-1}]\).  The Frobenius formula says that
  \(\mathrm{ind}_K^H(\mathrm{res}_K^H(y)\cdot x) = y \cdot
  \mathrm{ind}_K^H(x)\) for all \(x\in R(K) \otimes \Z[S^{-1}]\),
  \(y\in R(H) \otimes \Z[S^{-1}]\); this is a relation in
  \(\KK^G_S(\Cont(G/H),\Cont(G/H))\).  We are interested in
  \(x=p_{k,k}\), \(y=p_{n,k}\).  The induction of representations is
  computed most easily on the level of characters: there we simply
  map a function \(\chi\colon K\to \C\) to the function
  \(\chi'\colon H\to\C\) given by \(\chi'(h)=0\) for \(h\notin K\)
  and \(\chi'(h) = \abs{H:K}\cdot \chi(h)\) for \(h\in H\)
  because~\(H\) is Abelian.  Using Lemma~\ref{lem:character_psi}, we
  see that the induced character of~\(p_{k,k}\) is
  \(\abs{{H:K}}\cdot p_{n,k}\).  Therefore, \(\abs{{H:K}}^{-1}\) times
  the product of the restriction generator, \(p_{k,k}\) and the
  induction generator gives the idempotent~\(p_{n,k}\) in
  \(\KK^G_S(\Cont(G/H),\Cont(G/H))\).  Thus, the restriction and
  induction generators provide a
  Murray--von Neumann equivalence between~\(p_{n,k}\) and a certain
  subprojection of~\(p_{k,k}\) in
  \(\KK^G_S(\Cont(G/K),\Cont(G/K))\), as needed.
\end{proof}

\begin{corollary}
  The homological functors that combine \(\KK^G_S(\Cont(G/H),{-})\)
  and \(\KK^G_S(A_H^0,{-})\), respectively, for all cyclic subgroups
  \(H\subseteq G\), have the same kernels on morphisms.  As a
  consequence, they generate the same relative homological algebra.
\end{corollary}

\begin{proof}
  The kernel on morphisms does not change if we add a direct summand
  of a homological functor to a list of homological functors or if
  we leave out several objects that are direct sums of others in the
  list.  Using this repeatedly, the claim follows from
  Proposition~\ref{pro:fewer_generators_localisation}.
\end{proof}

The objects \(A_H^0\) for two conjugate cyclic subgroups are
isomorphic in \(\KK^G_S\) through conjugation.  Therefore, our
homological ideal does not change if we only take one cyclic subgroup
\(H\subseteq G\) in each conjugacy class.  The resulting generators
have the nice extra property that they are orthogonal, that is, any
element \(\KK^G(A^0_H, A^0_{H'})\) for \(H\) not conjugate to~\(H'\)
vanishes:

\begin{lemma}
  \label{lem:AH_zero_orthogonal}
  If there is a nonzero element in \(\KK^G_S(A^0_{H}, A^0_{H'})\) for
  two cyclic subgroups \(H,H' \subseteq G\), then~\(H\) is conjugate
  to~\(H'\).
\end{lemma}

\begin{proof}
  In the proof of Theorem~\ref{thm:ivo}
  in~\cite{dellAmbrogio:Cell_G}, it is shown that any element of
  \[
    \KK^G(\Cont(G/H),\Cont(G/H'))
  \]
  may be written as a sum of products
  \begin{multline*}
    \Cont(G/H) \to \Cont\bigl(G/(H\cap (H')^{g^{-1}})\bigr)
    \xrightarrow[\cong]{c_g} \Cont\bigl(G/(H^g\cap H')\bigr)
    \\\xrightarrow{m_E} \Cont\bigl(G/(H^g\cap H')\bigr)
    \to \Cont(G/H'),
  \end{multline*}
  where the first and last arrow are the generators that act by
  induction and restriction on \(\K^H_*(B)\), \(c_g\) is induced by right
  multiplication by~\(g\), and~\(m_E\) is obtained by induction from
  \(E \in \KK^{H^g\cap H'}(\C,\C) \cong R(H^g\cap H')\).  When we
  replace \(\Cont(G/H)\)
  and \(\Cont(G/H')\) by \(A_H^0\) and~\(A_{H'}^0\), respectively,
  then we multiply these composites on both sides by the idempotents
  \(p_{n,n}\) and \(p_{n',n'}\), where \(n\defeq \abs{H}\) and
  \(n'\defeq \abs{H'}\).  These idempotents restrict to~\(0\) in any
  proper subgroup.  Hence, by the Frobenius formula for Mackey
  modules in~\cite{dellAmbrogio:Cell_G}, these products kill the
  restriction and induction generators unless
  \(H\cap (H')^{g^{-1}} = H\) and \(H^g\cap H' = H'\) or,
  equivalently, \(H^g = H'\).  Therefore,
  \(\KK^G_S(A^0_{H}, A^0_{H'})\) vanishes unless \(H\) and~\(H'\)
  are conjugate.
\end{proof}

\begin{lemma}
  Let \(H\subseteq G\) be a cyclic subgroup.  Let
  \(N_H \defeq \{g\in G \mid g H g^{-1} = H\}\) and \(n\defeq \abs{H}\).
  The canonical action of~\(N_H\) on~\(H\) induces an action on its
  representation ring, which further induces an action on
  \(\Z[\vartheta_n,1/\abs{G}]\).  This action is trivial on
  \(H \subseteq N_H\) and therefore descends to an action of the
  quotient group \(W_H \defeq N_H/H\).  The endomorphism ring
  \(\KK^G_S(A^0_{H}, A^0_{H})\) is isomorphic to the crossed product ring
  \(\Z[\vartheta_n,1/\abs{G}] \rtimes W_H\).
\end{lemma}

\begin{proof}
  The direct summand in the localised representation ring of~\(H\)
  that is the
  image of~\(p_n\) is isomorphic to \(\Z[\vartheta_n,1/\abs{G}]\)
  by~\eqref{eq:cyclic_rep_localized}.  This summand is invariant
  under group automorphisms, so that~\(W_H\) acts naturally on it.
  The span of the elements corresponding to representations of $H$ in the
  endomorphism ring of~\(A^0_H\) in~\(\KK^G_S\) is isomorphic to
  \(\Z[\vartheta_n,1/\abs{G}]\).  Conjugation by any one of \(g\in
  N_H\) leaves this subspace invariant
  and acts by the canonical \(N_H\)\nb-action on
  \(\Z[\vartheta_n,1/\abs{G}]\) mentioned above.  An argument
  as in the proof of
  Lemma~\ref{lem:AH_zero_orthogonal} shows that products of
  elements corresponding to group representations and conjugations for \(g\in N_H/H\) span
  the endomorphism ring of~\(A^0_H\) in~\(\KK^G_S\) and satisfy the
  relations in the crossed product
  \(\Z[\vartheta_n,1/\abs{G}] \rtimes W_H\).  Theorem~\ref{thm:ivo}  shows that the canonical map from
  \(\Z[z]/(z^n-1) \rtimes W_H\) to the endomorphism ring of~\(\Cont(G/H)\)
  in~\(\KK^G\) is injective.  This remains so after
  inverting~\(\abs{G}\) because the groups involved are
  torsion-free.  And then it follows that the map from
  \(\Z[\vartheta_n,1/\abs{G}] \rtimes W_H\) to the endomorphism
  ring of~\(A^0_H\) in~\(\KK^G_S\) is injective.
\end{proof}

\begin{remark}
  Let~$X$ be a set of representatives for the conjugacy classes of
  cyclic subgroups of~$G$.  Recall that $A_H^0 \cong A_{H'}^0$ if $H$
  and~$H'$ are conjugate.  This and
  Proposition~\ref{pro:fewer_generators_localisation} imply that the
  endomorphism rings of \(\bigoplus_{H\in X} A^0_H\) and
  \(\bigoplus_{\text{cyclic }H\subseteq G} \Cont(G/H)\) in the
  localisation of~\(\KK^G\) at~\(S\) are Morita equivalent: each ring
  is isomorphic to a corner in a matrix algebra over the other ring.
  These rings usually fail, however, to be isomorphic.  The module
  category over the first ring is~\(\mathfrak{A}_G\), and the second
  ring is the localisation of~\(R^G_{\textup{cy}}\) at~$S$.
  Therefore, the categories
  \(\mathfrak{Mod}^{\Z/2}(R^G_{\textup{cy}})_{\aleph_1}[S^{-1}]\)
  and~\(\mathfrak{A}_G\) are equivalent, but not isomorphic.
\end{remark}

At this point, the general machinery of homological algebra in
triangulated categories shows that the universal Abelian
approximation for \(\KK^G_S\) with respect to the homological ideal
that we are looking at is the functor to the category of countable
\(\Z/2\)-graded modules over
\(\bigoplus_H \Z[\vartheta_n,1/\abs{G}] \rtimes W_H\) which maps an
object~\(B\) to the family \(\KK^G_{S,*}(A_H^0,B)\); here~\(H\) runs
through a set of representatives for the conjugacy classes of cyclic
subgroups in~\(G\).  Inspection shows that this functor is naturally
isomorphic to the
functor~\(F\) in Theorem~\ref{the:UCT_S}.  To get the Universal
Coefficient Theorem in that theorem, it remains to prove that the
target category is hereditary.

\begin{lemma}
  \label{lem:hereditary_WH}
  Let \(H\subseteq G\) be a cyclic subgroup, let \(W_H\) and~\(n\)
  be as in the previous lemma.  The crossed product ring
  \(\Z[\vartheta_n,1/\abs{G}] \rtimes W_H\) is hereditary, that is,
  any module over it has a projective resolution of length~\(1\).
\end{lemma}

\begin{proof}
  We claim that a module over \(\Z[\vartheta_n]\) is projective if and
  only if it is free as an Abelian group.  This is shown in the
  proof of \cite{Koehler:Thesis}*{Theorem~12.14}.  That theorem is
  only stated if~\(n\) is prime because that is the case that is
  needed by Köhler at the time.  The proof, however, only uses that
  \(\Z[\vartheta_p]\) for a prime~\(p\) is a Dedekind domain, and this
  remains true for all~\(n\).  As a consequence, any submodule of a
  free module over \(\Z[\vartheta_n]\) is projective.  Then it follows
  that any module~\(M\) over~\(\Z[\vartheta_n]\) has a projective
  resolution \(0 \to P_1 \to P_0 \to M\) of length~\(1\).  This
  remains a resolution if we tensor by \(\Z[\abs{G}^{-1}]\) because
  the latter is flat, and \(P_j \otimes_Z \Z[\abs{G}^{-1}]\) is
  projective as a module over \(\Z[\vartheta_n,1/\abs{G}]\) because
  \[
    \Hom(P_j \otimes_Z \Z[\abs{G}^{-1}], N) \cong \Hom(P_j, N)
  \]
  if~\(N\) is a module over \(\Z[\vartheta_n,1/\abs{G}]\).  As a
  consequence, any submodule of a projective module over
  \(\Z[\vartheta_n,1/\abs{G}]\) is itself projective.

  Since~\(\abs{W_H}\) divides~\(\abs{G}\), averaging over the
  group~\(W_H\) is possible after inverting~\(\abs{G}\).  Therefore,
  an extension of modules over
  \(\Z[\vartheta_n,1/\abs{G}] \rtimes W_H\) that splits by a
  \(\Z[\vartheta_n,1/\abs{G}]\)-module map also splits by a
  \(\Z[\vartheta_n,1/\abs{G}] \rtimes W_H\)-module map.  Thus, a
  module is projective over
  \(\Z[\vartheta_n,1/\abs{G}] \rtimes W_H\) once it is projective
  over the subring \(\Z[\vartheta_n,1/\abs{G}]\).  Now it follows that
  any submodule of a projective module over
  \(\Z[\vartheta_n,1/\abs{G}] \rtimes W_H\) is itself projective.  This
  is equivalent to our statement.
\end{proof}

Any countable module over a countable ring is a quotient of a
countable free module.  Therefore, Lemma~\ref{lem:hereditary_WH}
implies that the
category~\(\mathfrak{A}_G\) is hereditary.  Now the Universal
Coefficient in Theorem~\ref{the:UCT_S} follows from
Theorem~\ref{thm:trspectral}.  This implies
Theorem~\ref{the:UCT_in_introduction} by well known arguments, as
with the usual Universal Coefficient Theorem (see, for example,
\cite{Blackadar:K-theory}*{Section 23.10}).

By Proposition~\ref{pro:localised_bootstrap}, an object~\(B\) in the
\(G\)\nb-equivariant bootstrap class is \(S\)\nb-divisible if and
only if multiplication by~\(\abs{G}\) is invertible on
\(\KK^G(\Cont(G/H),B)\) for all cyclic subgroups~\(H\).  This makes it easier to check this hypothesis if
we already know that~\(B\) is in the equivariant bootstrap class.
We do not know a checkable necessary and sufficient criterion for
the latter, however.

Now we specialise to \emph{Kirchberg algebras}, that is, nonzero,
simple, purely infinite, nuclear \(\Cst\)\nb-algebras.  In that
case, we may lift the classification of actions up to
\(\KK^G\)-equivalence to a classification up to cocycle conjugacy:

\begin{theorem}[Gabe and Szabó~\cite{Gabe-Szabo:Dynamical_Kirchberg}]
  Let~\(G\) be a finite group.  Any \(G\)\nb-action on a separable,
  nuclear \(\Cst\)\nb-algebra is \(\KK^G\)-equivalent to a pointwise
  outer action on a stable Kirchberg algebra.  Two pointwise outer
  \(G\)\nb-actions on stable Kirchberg algebras are
  \(\KK^G\)-equivalent if and only if they are cocycle conjugate.
\end{theorem}

\begin{proof}
  The first claim is a special case of
  \cite{Meyer:Actions_Kirchberg}*{Theorem~2.1}.  The second claim is
  a special case of
  \cite{Gabe-Szabo:Dynamical_Kirchberg}*{Theorem~A}.
\end{proof}

\begin{corollary}
  \label{cor:classify_2-local}
  Let~\(G\) be a finite group.  There is a bijection between the set
  of isomorphism classes of objects of\/~\(\mathfrak{A}_G\) and the
  set of cocycle conjugacy classes of pointwise outer
  \(G\)\nb-actions on stable Kirchberg algebras that belong to the
  \(G\)\nb-equivariant bootstrap class and are \(S\)\nb-divisible in
  \(\KK^G\).
\end{corollary}

\begin{example}
  Let \(G=\Z/p\) be a cyclic group whose order is a prime
  number~\(p\).  Then~\(G\) only has the cyclic subgroups \(\{1\}\)
  and~\(G\).  So our invariant~\(F\) takes values in the product of
  the categories of \(\Z/2\)-graded countable
  \(\Z\rtimes G\)-modules and of \(\Z/2\)-graded countable
  \(\Z[\vartheta_p]\)-modules.  The crossed product \(\Z\rtimes G\)
  is the group ring of~\(G\), which is naturally isomorphic to
  \(\Z[z]/(z^p-1)\).  After inverting~\(p\), this splits as
  \(\Z[1/p]\oplus \Z[\vartheta_p,1/p]\)
  by~\eqref{eq:cyclic_rep_localized}.  Thus our classification
  theorem implies that isomorphism classes of objects in~\(\KK^G\) are
  in bijection with triples \((X,Y,Z)\) where \(X\) is a
  \(\Z/2\)-graded Abelian group and \(Y,Z\) are \(\Z/2\)-graded
  \(\Z[\vartheta_p,1/p]\)-modules.  This is equivalent to the
  classification that follows from K\"ohler's UCT and
  \cite{Meyer:Actions_Kirchberg}*{Theorem~7.2}.
\end{example}

\begin{example}
  Let \(V=\Z/2\times \Z/2\) be the Klein four-group with generators
  \(a,b\) subject to the relations \(a^2 = b^2 = (a b)^2 = 1\).
  Besides the trivial subgroup, this has exactly three cyclic
  subgroups, namely, those generated by \(a,b,a b\), and they are
  of order~\(2\).  For these three subgroups, we find
  \(\Z[\vartheta_2] = \Z\) and \(N_H/H \cong \Z/2\), acting
  trivially.  So the crossed product
  \(\Z[\vartheta_2,1/2] \rtimes \Z/2\) is isomorphic to
  \(\Z[1/2]^{\oplus 2}\) because of the two characters of~\(\Z/2\).
  For the trivial subgroup, the relevant ring is
  \(\Z[\vartheta_1,1/2]\rtimes V\), the group ring of~\(V\) with
  coefficients in \(\Z[1/2]\).  The evaluation at the four
  characters splits this group ring as \(\Z[1/2]^{\oplus 4}\).  So
  altogether, we get \(4+6=10\) summands \(\Z[1/2]\).  Thus
  isomorphism classes of \(2\)\nb-divisible objects in the
  \(V\)\nb-equivariant bootstrap class in~\(\KK^V\) are in bijection
  with \(10\)-tuples of \(2\)\nb-divisible \(\Z/2\)-graded Abelian
  groups.
\end{example}

\begin{remark}
  Let~\(B\) be any object in the \(G\)\nb-equivariant bootstrap
  class.  The inclusion map \(B \to B \otimes \Mat_{S^\infty}\) is
  part of an exact triangle \(B \to B \otimes \Mat_{S^\infty} \to
  B\otimes C \to \Sigma B\) with \(\K_1(C) = 0\) and
  \[
    \K_0(C) = \Z[S^{-1}]/\Z
    \cong \bigoplus_{p\in S} \Z[p^{-1}]/\Z,
  \]
  where the last isomorphism follows from the Chinese Remainder
  Theorem.  Thus we may decompose \(C = \bigoplus_{p\in S} C_p\)
  with \(\K_1(C_p)=0\), \(\K_0(C_p) = \Z[1/p]/\Z\).  We may
  write~\(B\) as the desuspended mapping cone of the map
  \(B \otimes \Mat_{S^\infty} \to \bigoplus_{p\in S} B\otimes C_p\)
  in the above exact triangle.  If we know this arrow in \(\KK^G\),
  we know~\(B\) up to isomorphism.  We may further write
  \(C_p = \varinjlim C_{p,n}\) where~\(C_{p,n}\) is the object in
  the bootstrap class with \(\K_1(C_{p,n})=0\) and
  \(\K_0(C_{p,n})=\Z/p^n\).  Thus \(B\otimes C_{p,n}\) has the extra
  property that multiplication by~\(p^n\) vanishes in its
  \(\KK^G\)-endomorphism ring.  Thus, one may try to classify
  general objects of the equivariant bootstrap class by first
  classifying those objects with the property that multiplication
  by~\(p^n\) vanishes in their \(\KK^G\)-endomorphism ring for some
  \(p\mid \abs{G}\), \(n\in\N\).  Even when this can be done, it still
  remains to classify the maps in \(\KK^G\) from an
  \(S\)\nb-divisible object to one of the form \(B\otimes C_p\).
\end{remark}

\end{document}